\documentclass[12pt,reqno]{amsart}
\usepackage{amsmath}
\usepackage[english,  activeacute]{babel}
\usepackage[latin1]{inputenc}
\usepackage{amssymb}
\usepackage{amsthm}
\usepackage{graphics,graphicx}
\usepackage{array}
\usepackage{a4wide}
\allowdisplaybreaks
\usepackage{color, url}
\usepackage{float}
\usepackage{multicol}
\usepackage[shortlabels]{enumitem}
\usepackage[hypertexnames=false,colorlinks=true,allcolors=blue]{hyperref}
\theoremstyle{plain}
\newtheorem{theorem}{Theorem}[section]

\newtheorem{lemma}[theorem]{Lemma}
\theoremstyle{definition}

\newtheorem{remark}[theorem]{Remark}

\begin{document}
		\title{Further results on arithmetic properties of biregular  overpartitions  } 
	
	\author{Suparno Ghoshal and Arijit Jana}
	\address{Department of Computer Science, Ruhr University Bochum, Germany}
	\email{suparno.ghoshal@rub.de}
	\address{Department of Mathematics, National Institute of Technology Silchar, Assam 788010, India}
	\email{jana94arijit@gmail.com}
	
	
	\begin{abstract}
		Recently there has been quite a bit of study carried out related to arithmetic properties of overpartitions into non-multiples of two co-prime integers. The paper \cite{Nadji} by Nadji et al. looked into congruences modulo $3$ and powers of $2$ for certain specific pairs of co-prime integers, while the paper \cite{AMS} by Alanazi et al. investigated some congruences related to some similar and some different pairs of co-prime integers. In this paper we propose some elegant and elementary proofs of a subset of the congruences given in \cite{AMS} by using only theta function and dissection identities. We also propose a generic method for proving congruences modulo $8$ which doesn't necessarily use any specific $2$-dissection. 
	\end{abstract}
	\maketitle
\section{Introduction}
The partition function $p(n)$ is usually defined as a function which counts the total number of legitimate partitions of a positive integer $n$. A ``legitimate" partition of a positive integer $n$ is defined as a non-increasing sequence of positive integers $\lambda=(\lambda_1, \lambda_2, \ldots, \lambda_k)$ which sums to $n$. Each $\lambda_i$ usually denotes "parts" of a partition. The generating function of $p(n)$ is written as 
\begin{equation}\label{eq0}
\sum_{n\geq 0}p(n)q^n = \prod\limits_{n=1}^{\infty }\frac{1}{1-q^n}= \frac{1}{f_1}.
\end{equation}
Here $q \in \mathbb{C}$  and $|q| < 1$. Throughout this paper we use the notation $f^k_n$ is defined by $f^k_n:=\prod_{j\geq 1}(1-q^{jn})^k$ for all integers $n,k$ with $n>0$.\\
Overpartitions are combinatorial structures usually associated with the $q$- Binomial theorem, Heine's transformation and Lebesgue's identity. An overpartition of $n$ is a partition of $n$ where the first occurrence of each part size may be overlined. We denote the number of overpartitions of $n$ by $\overline{p}(n)$, with $\overline{p}(0)=1$. For example, $\overline{p}(3)=8$, which enumerates the following overpartitions
$$(3),\, (\overline{3}),\, (2,1),\, (\overline{2},1),\ (2,\overline{1}),\, (\overline{2},\overline{1}),\, (1,1,1),\, (\overline{1},1,1).$$
The three overpartitions with no overlined parts are the ordinary partitions of 3.
We will also use the alternative notation of a partition: $\lambda=(c_1^{u_1},c_2^{u_2},\ldots,c_r^{u_r})$, where $c_1>c_2>\cdots >c_r>0$ and $r\leq k$. 
The following generating function for counting the number of overpartitions of a positive integer $n$ was introduced by Corteel and Lovejoy \cite{corteel1}
\begin{align}
\sum_{n = 0}^{\infty} \overline{p(n)}q^{n} = \prod_{n = 1}^{\infty} \dfrac{1 + q^{n}}{1 - q^{n}} = \dfrac{f_{2}}{f_{1}^{2}}.
\end{align}
There is an extensive literature studying the various arithmetic properties of the function $\overline{p(n)}$. One such well-known congruence was due to Hirschhorn and  Sellers \cite{hirs1}, where they had proved the following
\begin{align}
\overline{p}(27n + 18) \equiv 0 \pmod{3},\\
\overline{p}(27n) \equiv \overline{p}(3n) \pmod{3}.
\end{align}
For interest of the readers, we cite the following works \cite{kim1,kim2, xia1,xia2} which delves deeper into the area of \textit{overpartitions}. Several theorems and techniques used in partitions have been shown to have analogues in the study of overpartitions.\\

Given a positive integer $\ell > 1$ a partition $\lambda$ is called $\ell$-regular if no part of $\lambda$ is a multiple of $\ell$.
The generating function $b_{\ell}(n)$ denoting the $\ell$-regular partitions of $n$ is given by
\begin{align}
\sum_{n = 0}^{\infty} b_{\ell}(n)q^{n} = \dfrac{f_{l}}{f_{1}}.
\end{align}
The function $b_{5}(n)$ is odd only when $12n + 1$ is a perfect square was proved by Hirschhorn and Sellers \cite{hirs2}. An infinite family of congruences modulo $2$ was studied by Cui and Gu \cite{cui1} for $\ell$-regular partitions for some particular values of $\ell \in \{2, 4, 5, 8, 13, 16\}$. \\

Lovejoy \cite{lovejoy1} had studied the $\ell$- regular overpartition function which has the following generating function 
\begin{align}
\sum_{n = 0}^{\infty} \overline{A}_{\ell}(n)q^{n} = \dfrac{f_{l}^{2}f_{2}}{f_{1}^{2} f_{2l}}.
\end{align}
The generating function for $\ell$- regular overpartitions was given by Shen \cite{Shen}. In the same paper, Shen introduced some congruences for $\ell = 3$. Lovejoy in \cite{lovejoy1} had further given results related to overpartitions which were analogous to Gordon's \cite{Gordon1} famous generalization of the celebrated Rogers-Ramanujan identities. \\
Munagi and Sellers \cite{MunagiSellers2013} studied combinatorial and arithmetic properties of overpartitions of $n$ with $\ell$-regular overlined parts, denoted by $A_{\ell}(n)$. Alanazi and Munagi \cite{AlanaziMunagi} investigated the combinatorial identities for  $\ell$-regular overpartitions of $n$, denoted by $\overline{R_{\ell}}(n)$ which is the same as $\overline{A}_{\ell}(n)$. In addition, Alanazi et al. \cite{AlanaziAlenaziKeithMunagi} studied combinatorial and arithmetic properties of the overpartitions of $n$ with $\ell$-regular non-overlined parts. We also note that Alanazi et al. \cite{AlanaziMunagiSellers} had discussed various arithmetic properties of $\overline{R_{\ell}}(n)$. We also recall the generating function from \cite{AlanaziAlenaziKeithMunagi}
\begin{equation}\label{eq1}
\sum_{n\geq 0}\overline{R_\ell^\ast}(n)q^n = \prod\limits_{n=1}^{\infty }\frac{(1-q^{n\ell})(1+q^n)}{1-q^n} = \frac{f_2f_\ell}{f_1^2},
\end{equation}
where $\overline{R_\ell^\ast}(n)$ counts overpartitions of $n$ wherein non-overlined parts are $\ell$-regular and there are no restrictions on the overlined parts. Alanazi et al. \cite{AlanaziAlenaziKeithMunagi} proved several interesting results related to this function, some of which were recently extended by Sellers \cite{SellersBAMS}.

Recently in the papers \cite{AMS,Nadji},  the respective authors have delved deeper into the arithmetic properties of overpartitions which are simultaneously $\ell$-regular and $\mu$-regular with $\gcd(\ell, \mu) = 1$. This \cite{AMS} paper discussed proofs of several congruences related to $(\ell, \mu)$- regular overpartitions for certain $(\ell, \mu) \in \{(2, 3), (4, 3), (4, 9), (8, 27), (16, 81)\}$. The proof technique used in their paper for proving these congruences were based on some mathematica implementation of an algorithm of Radu \cite{Radu, Radu2} due to Smoot \cite{Smoot}. While in this \cite{Nadji} paper they have discussed congruences modulo 3 and powers of $2$ for certain $(\ell, \mu)$ pairs, more specifically $(\ell, \mu) \in \{(4, 3), (4, 9), (8, 3), (8, 9)\}$.  It was mentioned in the conclusion of this \cite{AMS} paper that it is desirable to have some elementary proofs of those congruences.  The main objective of our paper is to come up with elementary proofs of some of those congruences discussed in \cite{AMS} while also not providing proofs of similar nature for the congruences which were already proved in \cite{Nadji}. For example, in this paper we provide congruences for the pair $(\ell, \mu) = (8, 27)$ modulo $8$ which was not discussed in this \cite{Nadji} paper and although this \cite{AMS} paper have a proof of that, they make use of Radu's algorithm \cite{Radu, Radu2} for the proof which is independent of our proof strategy. We have mentioned in \textit{remarks} in each section specifically where we have extended the results in \cite{Nadji} and where we have reproved new congruences stated in \cite{AMS} by only using theta function or dissection identities.   \\
In this paper as we are concerned with $(\ell,\mu)$-biregular overpartitions of $n$, we go with the generating function which was used in these papers \cite{AMS}, \cite{Nadji}. Assume $\overline{R_{\ell,\mu}}(n)$ to be the number of overpartitions of $n$ where no parts are divisible by $\ell$ or $\mu$, where $\gcd (\ell,\mu)=1$. 
The generating function of such a function is as follows
\begin{equation}\label{eq2}
\sum_{n\geq 0}\overline{R_{\ell,\mu}}(n)q^n=\prod\limits_{n=1}^{\infty }\frac{ (
	1+q^{n})( 1-q^{\ell n})( 1-q^{\mu n})( 1+q^{\ell\mu n}) }{( 1-q^{n})( 1+q^{\ell n})( 1+q^{\mu n})( 1-q^{\ell\mu n}) } =\frac{f_2f_\ell^2f_\mu^2f_{2\mu\ell}}{f_1^2f_{2\ell}f_{2\mu}f_{\mu\ell}^2}.
\end{equation}

\subsection{Organization}
The paper is organized in the following manner. In Section \ref{preliminaries} we provide a brief overview of all the necessary results which we will require for proving our congruences in an elementary way. Further in Sections \ref{(2,3)_result}, \ref{(4,3)-result}, \ref{(4, 9)-result}, \ref{(8, 27)-result}, \ref{(16, 81)-result} we provide proofs for the congruences modulo $3 \text{ and } 8$ satisfied by the function $\overline{R_{\ell, \mu}}(n)$ for the following pairs $(\ell, \mu) \in \{(2, 3), (4, 3), (4, 9), (8, 27), (16, 81)\}$. We conclude in Section \ref{Conclude} with some remarks on our work and possible future work.

\section{Preliminaries}\label{preliminaries}
In this section we discuss some of the notations, definitions and important Lemmas and identities which will be used in the entire paper.
We begin by first stating the general mathematical definition of Ramanujan's theta function $f(a, b)$ \cite[Equation 1.2.1]{Spirit}
\begin{align*}
f(a,b) := \sum_{k=-\infty}^{\infty} a^{\frac{k(k+1)}{2}}b^{\frac{k(k-1)}{2}}, \quad |ab| < 1.
\end{align*}
Now we present the three important special cases of $f(a ,b)$ which we will need for proving our results and also for writing some of the generating functions. It is important to note that all of the product representations in the expressions below directly follow from Jacobi's triple product identity \cite[p. 35, Entry 19]{bcb3}
\begin{align}
\varphi(q) &:= f(q,q) = (-q;q^2)^2_\infty(q^2;q^2)_\infty=\dfrac{f_2^5}{f_1^2f_4^2},\label{varphi}\\
\psi(q) &:= f(q,q^3) = \frac{(q^2;q^2)_\infty}{(q;q^2)_\infty}=\dfrac{f_2^2}{f_1},\label{psi}\\
f(-q) &:= f(-q,-q^2) = (q;q)_\infty=f_1, \notag
\end{align}
Next we state a couple of $\chi$ function identities which are essential for discussing the Lemma \eqref{lemma:lem1} and also will be used in some of the proofs in the paper,
\begin{align}
\chi(q):=(-q;q^2)_\infty=\dfrac{f_2^2}{f_1f_4}, \label{chiq}\\
\chi(-q)=(q;q^2)_\infty=\dfrac{f_1}{f_2}, \label{chi-q} 
\end{align}

\noindent The identities below follow by replacing $q$ with $-q$ in \eqref{varphi} and \eqref{psi}
\begin{align}
\varphi(-q)&=\frac{f_1^2}{f_2},\label{varphi(-q)}\\
\psi(-q)&=\frac{f_1f_4}{f_2}.\label{psi(-q)}
\end{align}
Combining \eqref{psi} and \eqref{varphi(-q)} we get the following identities
\begin{align}\label{varphi psi}
\varphi(-q)\psi(q)=f_1 f_2.
\end{align}
Below we re-state two identities related to Ramanujan's theta function which we will use later in the paper,
\begin{lemma}\cite[p. 51 and p. 350]{bcb3}\label{lemma:lem1}
	We have
	\begin{align}
	\label{f(q,q^5)} f(q,q^5)&=\psi(-q^3)\chi(q),\\
	\label{f(q,q^2)}f(q,q^2)&=\frac{\varphi(-q^3)}{\chi(-q)},
	\end{align}
	
\end{lemma}

At this point we start presenting some of the important $2, 3$- dissection formulae. We discuss the $2$-dissection formulae in \eqref{lem-2-dis}.
\begin{lemma}\cite[Lemma 2]{matching}\label{lem-2-dis}
	We have
	\begin{align}
	f_1^4&=\dfrac{f_4^{10}}{f_2^2f_8^4}-4q\dfrac{f_2^2f_8^4}{f_4^2}\label{f_1^4},\\
	\frac{1}{f_{1}^4} &= \frac{f_{4}^{14}}{f_{2}^{14} f_{8}^4} + 4 q \frac{f_{4}^2 f_{8}^4}{f_{2}^{10}},\label{dis1byf1^4}
	\end{align}
\end{lemma}
The following identities are due to \cite[Lemma 2.1]{Guadalupe1}.
\begin{align}
\dfrac{f_3}{f_1^{3}}&= \frac{f_{4}^{6}f_{6}^{3} }{f_{2}^{9}f_{12}^{2}} +3q \frac{f_{4}^{2}f_{6} f_{12}^{2} }{f_{2}^{7}} \label{eq10}\\
\dfrac{f_3^{3}}{f_1}&= \frac{f_{4}^{3}f_{6}^{2} }{f_{2}^{2}f_{12}} +q \frac{ f_{12}^{3} }{f_{4}} \label{eq33}
\end{align}
Here we state some of the well-known 3-dissection formulae.
\begin{lemma}\cite[p. 49]{bcb3}
	We have
	\begin{align}
	\varphi(q)&=\varphi(q^9)+2qf(q^3,q^{15}),\label{3-dissec-varphi}\\
	\psi(q)&=f(q^3,q^6)+q\psi(q^9).\label{3-dissec-psi}
	\end{align}
\end{lemma}
Following up on the above $3$- dissection identities we present some $3$-dissections identities which were given by Hirschhorn \cite[(14.3.2), (14.3.3)]{hirscb}: 
\begin{align}
\dfrac{f_1^2}{f_2} &= \dfrac{f_9^2}{f_{18}}-2q\dfrac{f_3f_{18}^2}{f_6f_9},\label{eq31}\\
\dfrac{f_2^2}{f_1}& = \dfrac{f_6f_9^2}{f_3f_{18}}+q\dfrac{f_{18}^2}{f_9}. \label{eq18}
\end{align}
Toh \cite[Lemma 2.1, (2.1c)]{toh} came up with the following identity  
\begin{align}
\dfrac{f_2}{f_1f_4} &= \dfrac{f_{18}^9}{f_3^2f_9^3f_{12}^2f_{36}^3}+q\dfrac{f_6^2f_{18}^3}{f_3^3f_{12}^3}+q^2\dfrac{f_6^4f_9^3f_{36}^3}{f_3^4f_{12}^4f_{18}^3}.\label{eq32}
\end{align}
which represents a $3$-dissection of the generating function for the number of partitions of $n$ with distinct odd parts.
The following $3$-dissection of the generating function which counts the number of partitions of $n$ with distinct even parts is due to Andrews, Hirschhorn,
and Sellers \cite[Theorem 3.1]{ahs}
\begin{align}\label{eqAHS}
\dfrac{f_4}{f_1} = \dfrac{f_{12}f_{18}^4}{f_3^3f_{36}^2}+q\dfrac{f_6^2f_9^3f_{36}}{f_3^4f_{18}^2}+2q^2\dfrac{f_6f_{18}f_{36}}{f_3^3}
\end{align}
which follows from the binomial theorem.
\begin{lemma}\cite[Lemma 2.2]{Guadalupe1}\label{lem22}The following $3$-dissection holds:
	\begin{align}
	\dfrac{f_2}{f_1^2} &= \dfrac{f_6^4f_9^6}{f_3^8f_{18}^3}+2q\dfrac{f_6^3f_9^3}{f_3^7}+4q^2\dfrac{f_6^2f_{18}^3}{f_3^6}.\label{eq9}
	\end{align}
\end{lemma}
Throughout the paper we will make use of the congruence $f_k^{p}\equiv f_{pk} \pmod{p}$ for $\text{ prime } p \text{ and  }k\geq 1$.

\section{Congruences for $\overline{R_{2,3}}(n)$ }\label{(2,3)_result}
\begin{theorem}
	For all $n\geq 0$, we have
	\begin{align}
	\overline{R_{2,3}}(9n+6)&\equiv 0 \pmod 6.\label{cong-1}
	\end{align}
\end{theorem}
\begin{proof}[Proof of Equation \eqref{cong-1}]
	Putting $\ell=2, \mu=3$ in \eqref{eq2}, we have
	\begin{align}\label{p1eq1}
	\sum_{n\geq 0}\overline{R_{2,3}}(n)q^n&=\frac{f_{2}^{3}f_{3}^{2}f_{12}}{f_1^2f_{4}f_{6}^{3}} \notag \\
	&\equiv \frac{f_{3}^{2}}{f_{6}^{2}}\frac{f_{4}^{2}}{f_{1}^{2}} \pmod3.
	\end{align}
	The last congruence holds because $f_{12}\equiv f_{4}^{3}$ and $f_{2}^{3}\equiv f_{6}$ in modulo $3$.
	Now, putting the value of $\frac{f_{4}}{f_{1}}$ from \eqref{eqAHS} in \eqref{p1eq1} and  extracting the terms $q^{3n}$ and then replacing $q^3$ by $q$ we get 
	\begin{align}\label{p1eq2}
	\sum_{n\geq 0}\overline{R_{2,3}}(3n)q^n&\equiv \frac{f_{1}^{2}}{f_{2}^{2}} \bigg(\frac{f_{4}^{2} f_{6}^{8}}{f_{1}^{6} f_{12}^{4}}+ 4 q \frac{f_{2}^3 f_{3}^3 f_{12}^2}{f_{1}^{7} f_{6}} \bigg)\notag \\
	&\equiv  \frac{f_{4}^{2} f_{6}^{8}}{f_{2}^{2}f_{1}^{4} f_{12}^{4}} +4q \frac{f_{2} f_{3}^{3} f_{12}^2}{f_{1}^{5} f_{6}}\notag \\
	&\equiv  \frac{ f_{6}^{8}}{ f_{12}^{4}}  \frac{f_{4}^{2}}{f_{2}}\frac{f_{1}^{2}}{f_{2 } f_{3}^{2}} + 4q \frac{ f_{3} f_{12}^{2} f_{1} f_{2}}{ f_{6}} \pmod3 .
	\end{align}
	The last congruence holds because of $f_{1}^{6} \equiv f_{3}^{2}$. 
	Replacing  $q$ by $q^2$ in \eqref{eq18}
	\begin{align}\label{p1eq3}
	\dfrac{f_4^2}{f_2} = \dfrac{f_{12}f_{18}^2}{f_{6}f_{36}}+q^2 \dfrac{f_{36}^2}{f_{18}}.
	\end{align}
	With the help of \eqref{p1eq3}, \eqref{eq31}, we see the coefficients of the term $q^{3n+2}$ of $\frac{ f_{6}^{8}}{ f_{12}^{4}}  \frac{f_{4}^{2}}{f_{2}}\frac{f_{1}^{2}}{f_{2 } f_{3}^{2}}$ is
	\begin{align}\label{p1eq4}
	\frac{ f_{6}^{8}}{ f_{12}^{4}}  \frac{f_{36}^{2}}{f_{18}}\frac{f_{9}^{2}}{f_{18} f_{3}^{2}} \equiv f_{6}^{2} f_{12}^{2}f_{9}f_{3} \pmod3.
	\end{align}
	In the last congruence, we use the fact $ f_{18} \equiv f_{6}^{3}, f_{36}\equiv f_{12}^{3}, f_{9}\equiv f_{3}^{3}$. 
	With the help of \eqref{varphi psi}, \eqref{3-dissec-psi}, and \eqref{3-dissec-varphi} we see the coefficients of the term $q^{3n+2}$ of
	$4q \frac{ f_{3} f_{12}^{2} f_{1} f_{2}}{ f_{6}} $ is 
	\begin{align}\label{p1eq5}
	&4\frac{ f_{3} f_{12}^{2} }{ f_{6}}\big(\varphi(-q^9) \psi(q^{9})-2 f(-q^3, -q^{15}) f(q^3, q^6)\big) \notag \\
	&\equiv 4\frac{ f_{3} f_{12}^{2} }{ f_{6}}\big(\varphi(-q^9) \psi(q^{9})-2 \varphi(-q^9) \psi(q^{9})\big) \notag \\
	&\equiv -4 f_{6}^{2} f_{12}^{2}f_{9}f_{3} \pmod3.
	\end{align}
	In the first congruence, we use \eqref{f(q,q^5)} and \eqref{f(q,q^2)}. In the last congruenc, we use the fact that $ \varphi(-q^9) \psi(q^{9})= f_{9} f_{18}$, $f_{18}\equiv f_{6}^{3}$.
	Combining \eqref{p1eq2}, \eqref{p1eq4}, \eqref{p1eq5} together we get 
	\begin{align}\label{p1eq6}
	\sum_{n\geq 0}\overline{R_{2,3}}(9n+6)q^n&=0 \pmod3.
	\end{align}
	Again, \begin{align*}
	\sum_{n\geq 0}\overline{R_{2,3}}(n)q^n&=\frac{f_{2}^{3}f_{3}^{2}f_{12}}{f_1^2f_{4}f_{6}^{3}} \notag \\
	&\equiv 1 \pmod2
	\end{align*}
	because of $ f_3^2 \equiv f_6, f_6^2\equiv f_{12} $ in modulo $2.$
	As a result,  \begin{align}\label{p1eq7}
	\sum_{n\geq 0}\overline{R_{2,3}}(9n+6)q^n \equiv 0 \pmod2
	\end{align}
	We complete the proof because of \eqref{p1eq6} and \eqref{p1eq7}.
\end{proof}
\begin{remark}
	The paper \cite{AMS} did provide a proof of the above theorem but they made use of Radu's algorithm \cite{Radu, Radu2} due to Smoot \cite{Smoot} while we give an elementary proof. It was also discussed in this \cite{Nadji2} paper that there's a recent paper where Nadji and Ahmia have some possible results for this case but since there are no published version available of this work as per our knowledge, Theorem \ref{cong-1} was not proved using only theta function identities. 
\end{remark}
\section{Congruences for $\overline{R_{4,3}}(n)$ }\label{(4,3)-result}
\begin{theorem}\label{thm:rd}
	For all $n\geq 0$, we have
	\begin{align}
	\overline{R_{4,3}}(12n + 7)&\equiv 0 \pmod 8,\label{cong-13}\\
	\overline{R_{4,3}}(12n + 11)&\equiv 0 \pmod 8.\label{cong-14}
	\end{align}
\end{theorem}

\begin{proof}
	
	We can re-write the generating function of $(4, 3)$- regular overpartitions in the following way by using some theta function identities 
	\begin{align}\label{seq1}
	\sum_{n\geq 0}\overline{R_{4,3}}(n)q^n&= \frac{\varphi(-q^4)\varphi(-q^3)}{\varphi(-q)\varphi(-q^{12})} .
	\end{align}
	Since $\varphi(-q^{12})$ already consists of terms which are of the form $q^{12n}$, we need to only check the remaining product and look for terms of the form $q^{12n + r}, \text{ where } r \in \mathbb{Z}_{12}$ in order to prove our congruences. Now by using the identity $\frac{1}{\varphi(-q)} = \varphi(q)\varphi(q^2)^{2}\varphi(q^{4})^{4} \cdots$ and the fact that $\varphi(q^i)^j \equiv 1 \pmod{8} \text{ for all even } j \geq 4$ we can re-write equation \eqref{seq1} in the following manner
	\begin{align}\label{seq2}
	&\equiv\frac{ \varphi(q)\varphi(q^2)^{2}\varphi(-q^3)\varphi(-q^4)}{\varphi(-q^{12})} \pmod{8}
	\end{align}
	From here we ignore the denominator and only consider the numerator of \eqref{seq2} and try and search for terms of the form $q^{12n + r}$. We can write the numerator terms of \eqref{seq2} as
	\begin{align}\label{seq3}
	&\equiv \biggl(1 + 2 \sum_{n \geq 1}q^{n^2}\biggl) 
	\biggl(1 + 2 \sum_{n \geq 1}q^{2n^2}\biggl)^{2}
	\biggl(1 + 2 \sum_{n \geq 1}(-q^3)^{n^2}\biggl)
	\biggl(1 + 2 \sum_{n \geq 1}(-q^4)^{n^2}\biggl) \pmod{8}.
	\end{align}
	By further expanding the above sum in \eqref{seq3} and reducing it modulo $8$ we get this expression
	\begin{align}\label{seq4}
	&\equiv 1 + 2 \sum_{n \geq 1}q^{n^2} + 2 \sum_{n \geq 1}(-q^{4})^{n^2} + 2 \sum_{n \geq 1}(-q^{3})^{n^2}+ 4  \sum_{i, j \geq 1}(-1)^{j^2}q^{i^{2} + 4 j^{2}} + 4  \sum_{i, j \geq 1}(-1)^{j^2}q^{i^{2} +  3 j^{2}} \notag \\
	& + 4  \sum_{i, j \geq 1}(-1)^{i^2+j^2}q^{ 3 i^{2} + 4 j^{2}}+ 4 \sum_{n \geq 1}q^{2n^2} + 4\sum_{i,j \geq 1}q^{2(i^2+j^2)} \pmod{8}.
	\end{align}
	In order to prove Theorem \ref{thm:rd}, we need to look into the power series expansion from equation \eqref{seq4}. We can clearly see that the exponents of $q$ are either squares or $2, 3$- times a square number or the sum of two squares or something of the form $i^{2} + 3 j^{2}$ or of the form $3i^{2} + 4j^{2}$. Lets use the fact that the square of any integer can belong to either one of the following residue classes $\{0, 1, 4, 9\} \pmod{12}$ and clearly our $n$ doesn't fall in any of the quadratic residue classes modulo $12$. So the coefficients can only come from the terms $4  \sum_{i, j \geq 1}(-1)^{i^2+j^2}q^{ 3 i^{2} + 4 j^{2}}$ or $4  \sum_{i, j \geq 1}(-1)^{j^2}q^{i^{2} +  3 j^{2}}$ and for a particular exponent we need to count these two terms twice for extracting terms of the form $q^{12n + 7}$. While for extracting terms of the form $q^{12n + 11}$ we can clearly see even in these two expansions one will never come across terms of the form $q^{12n + 11}$ in the expansion. Hence we have shown that there are no exponents of $q$ of the form $12n +7$ or $12n + 11$, which concludes our proof.
	
\end{proof}

\begin{theorem}
	For all $n\geq 0$, we have
	\begin{align}
	\overline{R_{4,3}}(9n+3)&\equiv 0 \pmod 6,\label{cong-2}\\
	\overline{R_{4,3}}(12n+11)&\equiv 0 \pmod{9}.\label{cong-4}
	\end{align}
\end{theorem}
\begin{proof}[Proof of Equation \eqref{cong-2}]
	Putting $\ell=4, \mu=3$ in \eqref{eq2}, we have
	\begin{align}\label{p1eq8}
	\sum_{n\geq 0}\overline{R_{4,3}}(n)q^n&=\frac{f_{2}f_{4}^{2} f_{3}^{2}f_{24}}{f_1^2f_{8}f_{6} f_{12}^{2}} \notag \\
	&\equiv \frac{f_{3} f_{24}}{f_{6}f_{12}^{2}} \frac{f_{1}f_{2} f_{4}^{2}}{f_{8}} \pmod3 \notag \\
	& = \frac{f_{3} f_{24}}{f_{6}f_{12}^{2}} \psi(q) \varphi(-q) \varphi(-q^4) \pmod3 .
	\end{align}
	
	Using \eqref{3-dissec-varphi} and  \eqref{3-dissec-psi},  we deduce
	\begin{align*}
	\psi(q) \varphi(-q) \varphi(-q^4)= &\big(f(q^3,q^6)+ q\psi(q^9) \big) \big(\varphi(-q^9)-2qf(-q^3,-q^{15})\big)\\  \quad &\big(\varphi(-q^{36})-2q^4f(-q^{12},-q^{60})\big).
	\end{align*}
	Putting the above value of in \eqref{p1eq8} and  extracting the terms of the form $q^{3n}$ and then replacing $q^3$ by $q$ we get 
	\begin{align}\label{p1eq9}
	\sum_{n\geq 0}\overline{R_{4,3}}(3n)q^n&\equiv \frac{f_{1} f_{8}}{f_{2}f_{4}^{2}}  \big[ f(q,q^2) \varphi(-q^3) \varphi(-q^{12})+q^2 \psi(q^3)f(-q,-q^{5})f(-q^{4},-q^{20}) \big]. \notag \\
	\end{align}
	
	We see that 
	\begin{align}\label{p1eq11}
	&\frac{f_{1} f_{8}}{f_{2}f_{4}^{2}}  \bigg( f(q,q^2) \varphi(-q^3) \varphi(-q^{12})\bigg) \notag \\
	&=\frac{f_{1} f_{8}}{f_{2}f_{4}^{2}}  \bigg(\varphi(-q^3) \frac{f_{2}}{f_{1}} \varphi(-q^3) \varphi(-q^{12})\bigg) \notag \\
	&=\frac{f_{12}^{2} f_{8}}{f_{24}f_{4}^{2}} \frac{f_{3}^{4}}{f_{6}^{2}}  \notag \\
	&= \frac{f_{12}^{2} f_{3}^{4}}{f_{24}f_{6}^{2}} \bigg( \dfrac{f_{24}^4f_{36}^6}{f_{12}^8f_{72}^3}+2q^{4}\dfrac{f_{24}^3f_{36}^3}{f_{12}^7}+4q^8\dfrac{f_{24}^2f_{72}^3}{f_{12}^6} \bigg).
	\end{align}
	The last equality holds because of \eqref{eq9}.  Extracting terms of the form $q^{3n+1}$ of \eqref{p1eq11} we have the following 
	\begin{align}\label{p1eq12}
	& 2q^4\frac{f_{12}^{2} f_{3}^{4}}{f_{24}f_{6}^{2}}\frac{f_{24}^{3}f_{36}^{3}}{f_{12}^{7}} \notag \\
	&= 2q^4 \frac{ f_{3}^{4} f_{24}^{2} f_{36}^{3}}{f_{6}^{2}f_{12}^{5}}.
	\end{align}
	
	Further note that for the second part we have, 
	\begin{align}\label{p1eq13}
	&\frac{f_{1} f_{8}}{f_{2}f_{4}^{2}}  \bigg( q^2 \psi(q^3)f(-q,-q^{5})f(-q^{4},-q^{20})\bigg) \notag \\
	&=q^2\frac{f_{1} f_{8}}{f_{2}f_{4}^{2}}  \bigg(\psi(q^3)^{2}\chi(-q) \psi(q^{12} \chi(-q^4)\bigg) \notag \\
	&\equiv q^2 \frac{f_{6}^{4} f_{24}^{2}}{f_{12}f_{4}f_{2}^{2}} \frac{1}{f_{1}^{4}}   \pmod3 \notag\\
	&\equiv q^2 \frac{f_{6}^{3} f_{24}^{2}}{f_{12}f_{3}} \frac{f_{2}}{f_{1}f_{4}}   \pmod3 .
	\end{align}
	By using the $3$-dissection  \eqref{eq32} and collecting terms of the form $q^{3n+1}$ from the above expression, we get
	\begin{align}\label{p1eq14}
	& q^4 \frac{f_{6}^{3} f_{24}^{2}}{f_{12}f_{3}} \frac{f_{6}^{4} f_{9}^{3} f_{36}^{3}}{f_{3}^{4}f_{12}^{4}f_{18}^{3}} \notag\\
	&= q^4 \frac{f_{6}^{7} f_{24}^{2}f_{9}^{3}f_{36}^{3}}{f_{12}^{5}f_{3}^{5}f_{18}^{3}}.
	\end{align}
	Using the binomial theorem we know that $f_{6}^{9} \equiv f_{18}^{3}$ and hence we can further reduce the above equation as the following
	\begin{align}\label{p1eq15}
	&\equiv q^4\frac{f_{24}^{2}f_{3}^{4}f_{36}^{3}}{f_{12}^{5}f_{6}^{2}}\pmod{3}.
	\end{align}
	Combining \eqref{p1eq9}, \eqref{p1eq12}, \eqref{p1eq15} together we deduce that
	\begin{align}\label{p1eq16}
	\sum_{n\geq 0}\overline{R_{4,3}}(9n + 3)q^n&=0 \pmod3.
	\end{align}
	
	Also, it is very easy to see that, 	\begin{align*}
	\sum_{n\geq 0}\overline{R_{4,3}}(n)q^n&=\frac{f_{2}f_{4}^{2} f_{3}^{2}f_{24}}{f_1^2f_{8}f_{6} f_{12}^{2}} \notag \\
	&\equiv 1 \pmod2 .
	\end{align*}
	As a result, 	\begin{align}\label{p1eq17}
	\sum_{n\geq 0}\overline{R_{4,3}}(9n + 3)q^n&=0 \pmod2.
	\end{align}
	Equations \eqref{p1eq16}, \eqref{p1eq17} together complete the proof.
\end{proof}
\begin{proof}[Proof of Equation \eqref{cong-4}]
	We have
	\begin{align}\label{p1eq18}
	\sum_{n\geq 0}\overline{R_{4,3}}(n)q^n&=\frac{f_{2}f_{4}^{2} f_{3}^{2}f_{24}}{f_1^2f_{8}f_{6} f_{12}^{2}} \notag \\
	&= \frac{f_{3}^{2} f_{24}}{f_{6}f_{12}^{2}} \frac{f_{2}}{f_{1}^{2}} \frac{ f_{4}^{2}}{f_{8}}  \notag \\
	& =\frac{f_{3}^{2} f_{24}}{f_{6}f_{12}^{2}} \bigg(\dfrac{f_6^4f_9^6}{f_3^8f_{18}^3}+2q\dfrac{f_6^3f_9^3}{f_3^7}+4q^2\dfrac{f_6^2f_{18}^3}{f_3^6}\bigg)\bigg(\varphi(-q^{36})-2q^{4}f(-q^{12},-q^{60}) \bigg).
	\end{align}
	The last equality holds for \eqref{eq9} and  \eqref{3-dissec-varphi}.
	Extracting the terms of the form $q^{3n+2}$ and after that dividing by $q^2$ and then  replacing $q^3$ by $q$ we get 
	\begin{align}\label{p1eq19}
	\sum_{n\geq 0}\overline{R_{4,3}}(3n+2)q^n&\equiv \frac{f_{1}^{2} f_{8}}{f_{2}f_{4}^{2}}  \big[ 4 \frac{f_{2}^{2} f_{6}^{3}}{f_{1}^{6}} \varphi(-q^{12}) -4q \frac{f_{2}^{3} f_{3}^{3}}{f_{1}^{7}} f(-q^4, -q^{20}) \big]. 
	\end{align}
	
	We see that 
	\begin{align}\label{p1eq20}
	& \frac{f_{1}^{2} f_{8}}{f_{2}f_{4}^{2}}\bigg( 4 \frac{f_{2}^{2} f_{6}^{3}}{f_{1}^{6}} \varphi(-q^{12}) \bigg) \notag \\
	&= 4 \frac{f_{2} f_{6}^{3}}{f_{4}^{2}} \varphi(-q^{12})\frac{1}{f_{1}^{4}} \notag \\
	&= 4 \frac{f_{2} f_{6}^{3}}{f_{4}^{2}} \varphi(-q^{12})  \bigg(\frac{f_{4}^{14}}{f_{2}^{14} f_{8}^4} + 4 q \frac{f_{4}^2 f_{8}^4}{f_{2}^{10}}\bigg).
	\end{align}
	The last equality holds because od \eqref{dis1byf1^4}.
	Extracting terms of the form $q^{2n+1}$ from \eqref{p1eq20}, and then dividing by $q$ and after that replace $q^2$ by $q$  we have the following 
	\begin{align}\label{p1eq21}
	& 16 \varphi(-q^{6}) \frac{f_{3}^3 f_{4}^4}{f_{1}^{9}} \notag \\
	&=  16 \varphi(-q^{6})f_{4}^4 \bigg(\frac{f_{3} }{f_{1}^{3}}\bigg)^3\notag \\
	&=16 \varphi(-q^{6})f_{4}^4 \bigg( \frac{f_{4}^{6}f_{6}^{3} }{f_{2}^{9}f_{12}^{2}} +3q \frac{f_{4}^{2}f_{6} f_{12}^{2} }{f_{2}^{7}}\bigg)^{3} \notag \\
	&\equiv 16 \varphi(-q^{6})f_{4}^4  \frac{f_{4}^{18}f_{6}^{9} }{f_{2}^{27}f_{12}^{6}} \pmod9 .
	\end{align}
	The last equality holds because of \eqref{eq10}.
	We see that there are no terms of the form $q^{2n+1}$ in \eqref{p1eq21}. 
	Further note that for the second part we have, 
	\begin{align}\label{p1eq22}
	\frac{f_{1}^{2} f_{8}}{f_{2}f_{4}^{2}}  \bigg( -4q \frac{f_{2}^{3} f_{3}^{3}}{f_{1}^{7}} f(-q^4, -q^{20}) \bigg) \notag \\
	= -4q f(-q^4, -q^{20}) \frac{f_{2}^{2} f_{8}}{f_{4}^{2}}  \bigg( \frac{f_{3} }{f_{1}^{3}}\bigg)^{3} f_{1}^{4}.
	\end{align}
	Using \eqref{eq10}, and \eqref{f_1^4} in \eqref{p1eq22}, then  collecting terms of the form $q^{2n+1}$ from the above expression, dividing by $q$, and then replacing $q^2$ by $q$  we get
	\begin{align}\label{p1eq23}
	&\equiv  -4 f(-q^2, -q^{10}) \frac{f_{1}^{2} f_{4}}{f_{2}^{2}} \frac{f_{2}^{28} f_{3}^{9}}{f_{1}^{29}f_{6}^{6}f_{4}^{4}} \pmod9\notag\\
	&=   -4 f(-q^2, -q^{10}) \frac{f_{2}^{26} }{f_{6}^{6}f_{4}^{3}} \bigg(\frac{f_{3}}{f_{1}^{3}}\bigg)^{9} \notag \\
	&= -4 f(-q^2, -q^{10}) \frac{f_{2}^{26} }{f_{6}^{6}f_{4}^{3}} \bigg(   \frac{f_{4}^{6}f_{6}^{3} }{f_{2}^{9}f_{12}^{2}} +3q \frac{f_{4}^{2}f_{6} f_{12}^{2} }{f_{2}^{7}}\bigg)^{9}\notag \\
	& \equiv -4 f(-q^2, -q^{10}) \frac{f_{2}^{26} }{f_{6}^{6}f_{4}^{3}} \bigg(   \frac{f_{4}^{6}f_{6}^{3} }{f_{2}^{9}f_{12}^{2}} \bigg)^{9} \pmod 9.
	\end{align}
	Note that, there are no terms of the of the form $q^{2n+1}$ in \eqref{p1eq23}.

	Combining \eqref{p1eq19}, \eqref{p1eq21}, \eqref{p1eq23} together we deduce that
	\begin{align}\label{p1eq16a}
	\sum_{n\geq 0}\overline{R_{4,3}}(12n + 11)q^n&=0 \pmod9.
	\end{align}
	
\end{proof}
\begin{remark}
	Proof of  \eqref{cong-13} is an extension of the result $\overline{R}_{4, 3}(12n + 7) \equiv 0 \pmod{4}$ which appeared in this paper \cite{Nadji}, while the proof of equations \eqref{cong-14}, \eqref{cong-2}, \eqref{cong-4} are proofs independent of the usage of Radu's algorithm \cite{Radu, Radu2}. 
\end{remark}
\section{Congruences for $\overline{R_{4,9}}(n)$ }\label{(4, 9)-result}
\begin{theorem}
	For all $n\geq 0$, we have
	\begin{align}
	\overline{R_{4,9}}(12n + 3)&\equiv 0 \pmod 8,\label{cong-10}\\
	\overline{R_{4,9}}(12n + 7)&\equiv 0 \pmod 8.\label{cong-11}\\
	\overline{R_{4,9}}(12n + 11)&\equiv 0 \pmod 8.\label{cong-12}
	\end{align}
\end{theorem}
\begin{proof}[Proofs of Equation \eqref{cong-10}, \eqref{cong-11}, \eqref{cong-12}]
	The generating function of $(4, 9)$- regular overpartitions can be written as 
	\begin{align}\label{pleq30}
	\sum_{n\geq 0}\overline{R_{4,9}}(n)q^n&= \frac{\varphi(-q^4)\varphi(-q^9)}{\varphi(-q)\varphi(-q^{36})} .
	\end{align}
	Since $\varphi(-q^{36})$ already consists of terms which are of the form $q^{12n}$, we need to only check the remaining product and look for terms of the form $q^{12n + r}, \text{ where } r \in \mathbb{Z}_{12}$ in order to prove our congruences. We can re-write $\frac{\varphi(-q^4)\varphi(-q^9)}{\varphi(-q)}$ in the following manner
	\begin{align}\label{pleq31}
	&\equiv \varphi(-q^4)\varphi(-q^9)\varphi(q)\varphi(q^2)^2 \pmod{8}\notag\\
	&\equiv \biggl(1 + 2 \sum_{n \geq 1}(-q^{4})^{n^2}\biggl) \biggl(1 + 2 \sum_{n \geq 1}(-q^{9})^{n^2}\biggl)\biggl(1 + 2 \sum_{n \geq 1}q^{n^2}\biggl) 
	\biggl(1 + 2 \sum_{n \geq 1}q^{2n^2}\biggl)^{2} \pmod{8} \notag \\
	&\equiv 1 + 2 \sum_{n \geq 1}q^{n^2} + 2 \sum_{n \geq 1}(-q^{4})^{n^2} + 2 \sum_{n \geq 1}(-q^{9})^{n^2} +
	4 \sum_{n \geq 1}q^{2n^2} +4\sum_{i,j \geq 1}q^{2(i^2+j^2)}\notag\\
	& + 4  \sum_{i, j \geq 1} (-1)^{j^2}q^{i^{2} + 4 j^{2}} + 4  \sum_{i, j \geq 1}(-1)^{i^2+j^2}q^{4i^{2} +  9 j^{2}} + 4  \sum_{i, j \geq 1}(-1)^{j^2}q^{i^{2} +  9 j^{2}}  \pmod{8}.
	\end{align}
	In the first congruence, we mainly use the fact $\frac{1}{\varphi(-q)} = \varphi(q)\varphi(q^2)^{2}\varphi(q^{4})^{4} \cdots$ and  $\varphi(q^i)^j \equiv 1 \pmod{8} \text{ for all even } j \geq 4$.
	In order to prove  \eqref{cong-10}, \eqref{cong-11}, and \eqref{cong-12}, we need to consider the fact that for our case $4 \text{ and } 9$ are both squares. By looking into the power series expansion from equation \eqref{pleq31} we can clearly see that the exponents of $q$ are either squares or $2$- times a square number or the sum of two squares. We can further use the fact that the square of any integer can belong to either one of the following residue classes $\{0, 1, 4, 9\} \pmod{12}$ and hence we can claim that there are no exponents of $q$ of the form $12n + 3$ or $12n +7$ or $12n + 11$, which proves our theorem.
\end{proof}
\begin{theorem}
	For all $n\geq 0$, we have
	\begin{align}
	\overline{R_{4,9}}(18n + 12)&\equiv 0 \pmod 3,\label{cong-20}\\
	\overline{R_{4,9}}(18n + 15)&\equiv 0 \pmod 3.\label{cong-21}
	\end{align}
\end{theorem}
\begin{proof}[Proofs of Equation \eqref{cong-20}, \eqref{cong-21}]
	As we know the generating function of $(4, 9)$- regular overpartitions can be written in the following way
	\begin{align}\label{seq6}
	\sum_{n\geq 0}\overline{R_{4,9}}(n)q^n&= \frac{\varphi(-q^4)\varphi(-q^9)}{\varphi(-q)\varphi(-q^{36})} ,
	\end{align}
	we use this to prove the two concerned congruences. First, we begin by extracting terms of the form $q^{9n + 3}$ from both sides of equation \eqref{seq6}. One thing to notice is that in order to extract terms of the form $q^{9n + 3}$ from the right side of equation \eqref{seq6}, we need to only concern ourselves with the fraction $\frac{\varphi(-q^4)}{\varphi(-q)}$. We begin by using the $3$- dissections in \eqref{eq9} and \eqref{3-dissec-varphi} and extracting terms of the form $q^{3n}$ from the expression $\frac{\varphi(-q^4)}{\varphi(-q)}$ we get
	\begin{align}\label{seq7new}
	&\equiv \dfrac{ \varphi(-q^{36}) f_{9}^{6} f_{6}^{4}  }{f_3^{8}  f_{18}^{3}} - 8 q^{6} \dfrac{f(-q^{12}, -q^{60})f_{6}^{2} f_{18}^{3}}{f_{3}^{6}} \pmod{3}.
	\end{align}
	Now we replace $q^{3}$ with $q$ in \eqref{seq7new} we get,
	\begin{align}\label{seq7new1}
	&\equiv \dfrac{ \varphi(-q^{12}) f_{2}^{4} f_{3}^{6}  }{f_1^{8}  f_{6}^{3}} - 8 q^{2} \dfrac{f(-q^{4}, -q^{20})f_{6}^{3} f_{2}^2}{f_{1}^{6}} \pmod{3}.
	\end{align}
	In the above equation \eqref{seq7new1}, we use the identity in \eqref{f(q,q^5)} (we set $q$ to $-q^{4}$ in the identity) and the known fact that $f_{k}^{3} \equiv f_{3k} \pmod{3}$ we get 
	\begin{align}\label{seq7}
	&\equiv \dfrac{ \varphi(-q^{12}) f_{2} f_{3}^{4}  }{f_1^{2}  f_{6}^{2}} - 8 q^{2} \dfrac{\psi(q^{12})f_{6}^{4} f_{4}}{f_{3}^{2}f_{8}f_{2}} \pmod{3}.
	\end{align}
	Now we further expand the above equation \eqref{seq7} and extract terms of the form $q^{3n + 1}$. For the expansion we make use of the $3$- dissections in \eqref{eq9} and \eqref{eq32} and we get
	
	\begin{align}
	&\equiv 2q\dfrac{\varphi(-q^{12})f_{6}f_{9}^{3}}{f_{3}^{3}} - 8q^{4}\dfrac{f_{6}f_{12}f_{36}^{3}}{f_{3}^{2}f_{24}} \pmod{3}.
	\end{align}
	From here we again first divide the expression by $q$ and then replace $q^{3}$ by $q$ and finally get the following expression
	\begin{align}\label{seq8}
	&\equiv -\dfrac{\varphi(-q^{4}) f_{2} f_{3}^{3}}{f_{1}^{3}} - 8q \dfrac{f_{2}f_{4}f_{12}^{3}}{f_{1}^{2}{f_{8}}} \pmod{3}.
	\end{align}
	At this stage we are in a position where we can show how equation \eqref{seq6} looks once we have extracted terms of the form $q^{9n + 3}$ from both side
	\begin{align}\label{seq9}
	\sum_{n\geq 0}\overline{R_{4,9}}(9n + 3)q^n &\equiv -\dfrac{\varphi(-q^{4}) \varphi(-q) f_{2} f_{3}^{3}}{\varphi(-q^{4})f_{1}^{3}} - 8q \dfrac{\varphi(-q)f_{2}f_{4}f_{12}^{3}}{\varphi(-q^{4})f_{1}^{2}{f_{8}}} \pmod{3}.
	\end{align}
	Further simplification of the above equation \eqref{seq9} using the identity \eqref{varphi(-q)} give us
	\begin{align}\label{seq10}
	&= -\dfrac{f_{3}^{3}}{f_{1}} - 8 q \dfrac{f_{12}^{3}}{f_{4}}.
	\end{align}
	
	From equation \eqref{seq10} we extract terms of the form $q^{2n + 1}$ by using the $2$- dissections in \eqref{eq33} and get the following 
	\begin{align}\label{seq11}
	&= -q \dfrac{f_{12}^{3}}{f_{4}}-8q \dfrac{f_{12}^{3}}{f_{4}} \notag \\
	&= -9q \dfrac{f_{12}^{3}}{f_{4}}.
	\end{align}
	This proves the congruence in \eqref{cong-20}.
	
	For proving the congruence in \eqref{cong-21}, we adhere to a similar strategy as was used in proving the previous congruence \eqref{cong-20} by first trying to extract terms of the form $q^{3n}$ which we already have from \eqref{seq7} and now instead of extracting terms of the form $q^{3n + 1}$, this time we go for extracting the terms of the form $q^{3n + 2}$ from equation \eqref{seq7}. For doing that again we use the $3$- dissections in \eqref{eq9} and \eqref{eq32} and then collect the terms of the form $q^{3n + 2}$ which leads us to the following expression
	\begin{align}\label{seq11new}
	\equiv 4q^{2}\dfrac{\varphi(-q^{12})f_{18}^{3}f_{3}^{4}}{f_{3}^{6}} - 8q^{2}\dfrac{\psi(q^{12})f_{6}^{2}f_{36}^{9}}{f_{3}^{2}f_{18}^{3}f_{24}^{2}f_{72}^{3}}.
	\end{align}
	Simplyfying equation \eqref{seq11new} by using the identity in \eqref{psi}, we get
	\begin{align}\label{seq11new1}
	\equiv 4q^{2}\dfrac{\varphi(-q^{12})f_{18}^{3}f_{3}^{4}}{f_{3}^{6}} - 8q^{2}\dfrac{f_{6}^{2}f_{36}^{9}}{f_{3}^{2}f_{12}f_{18}^{3}f_{72}^{3}}.
	\end{align}
	Now we divide by $q^{2}$ and replace $q^{3}$ by $q$ to get the following expression
	\begin{align}\label{seq12}
	&\equiv 4 \dfrac{\varphi(-q^{4})f_{6}^{3}}{f_{1}^{2}} - 8 \dfrac{f_{2}^{2} f_{12}^{9}}{f_{4} f_{1}^{2}f_{6}^{3}f_{36}^{3}} \pmod{3}.
	\end{align}
	Similarly as before we try and re-write equation \eqref{seq6} by extracting terms of the form $q^{9n + 6}$ from both sides and replacing $q^9$ by $q$
	\begin{align}\label{seq13}
	\sum_{n\geq 0}\overline{R_{4,9}}(9n + 6)q^n &\equiv 4\dfrac{\varphi(-q^{4}) \varphi(-q) f_{2} f_{6}^{3}}{\varphi(-q^{4})f_{1}^{2}} - 8 \dfrac{\varphi(-q)f_{2}^{2}f_{4}f_{12}^{9}}{\varphi(-q^{4})f_{1}^{2}{f_{4} f_{6}^{3}f_{36}^{3}}} \pmod{3}.
	\end{align}
	Further simplification of equation \eqref{seq13} by using the theta function identity in \eqref{varphi(-q)} leads us to the following expression
	\begin{align}\label{seq14}
	\sum_{n\geq 0}\overline{R_{4,9}}(9n + 6)q^n &\equiv 4\dfrac{f_{6}^{3}}{f_{2}} - 8 \dfrac{f_{2}f_{12}^{9}}{\varphi(-q^{4}){f_{4} f_{6}^{3}f_{36}^{3}}} \pmod{3}.
	\end{align}
	It can be easily seen that in equation \eqref{seq14} there are no terms of the form $q^{2n + 1}$ on right side of the equation which immediately implies the congruence \eqref{cong-21}.
\end{proof}
\begin{remark}
	Re-proved the congruences from this \cite{AMS} paper by simply using dissection and theta function identities. Although the results in equations \eqref{cong-10}, \eqref{cong-11}, \eqref{cong-12} are special cases of the congruence $\overline{R}_{4, 9}(4n + 3) \equiv 0 \pmod{8}$ in \cite{Nadji}, our proof method differs from theirs as we only use the summation formula of $\varphi(q)$ and argue about the residue classes modulo $12$ of the exponents in the power series expansion.
\end{remark}
\section{Congruences for $\overline{R_{8,27}}(n)$ }\label{(8, 27)-result}

\begin{theorem}
	For all $n\geq 0$, we have
	\begin{align}
	\overline{R_{8,27}}(36n + 15)&\equiv 0 \pmod {8}.\label{cong-24}
	\end{align}
\end{theorem}

\begin{proof}
	
	We have
	\begin{align}\label{p1eq30}
	\sum_{n\geq 0}\overline{R_{8,27}}(n)q^n&= \frac{\varphi(-q^8)\varphi(-q^{27})}{\varphi(-q)\varphi(-q^{216})} 
	\end{align}
	Since $\varphi(-q^{216})$ already consists of terms which are of the form $q^{36n}$, we need to only check the remaining product and look for terms of the form $q^{36n + r}, \text{ where } r \in \mathbb{Z}_{36}$. In order to prove our congruences, we can re-write equation \eqref{p1eq30} in the following manner:
	\begin{align}\label{p1eq31}
	&\equiv\frac{ \varphi(-q^8)\varphi(-q^{27})\varphi(q)\varphi(q^2)^{2}}{\varphi(-q^{216})} \pmod{8}
	\end{align}
	because of the fact that  $\frac{1}{\varphi(-q)} = \varphi(q)\varphi(q^2)^{2}\varphi(q^{4})^{4} \cdots$ and  $\varphi(q^i)^j \equiv 1 \pmod{8}$  for all even $ j \geq 4$.
	From here, we ignore the denominator and only consider the numerator of \eqref{p1eq31} and try and search for terms of the form $q^{36n + r}$:
	\begin{align}\label{seq3new}
	&\equiv 
	\biggl(1 + 2 \sum_{n \geq 1}(-q^8)^{n^2}\biggl)
	\biggl(1 + 2 \sum_{n \geq 1}(-q^{27})^{n^2}\biggl)\biggl(1 + 2 \sum_{n \geq 1}q^{n^2}\biggl) 
	\biggl(1 + 2 \sum_{n \geq 1}q^{2n^2}\biggl)^{2} \pmod{8}
	\end{align}
	By further expanding the above sum in \eqref{seq3new} and reducing it modulo $8$ we get this expression
	\begin{align}\label{seq4new}
	&\equiv 1 + 2 \sum_{n \geq 1}q^{n^2} + 2 \sum_{n \geq 1}(-q^{8})^{n^2} + 2 \sum_{n \geq 1}(-q^{27})^{n^2}+ 4  \sum_{i, j \geq 1}(-1)^{j^2}q^{i^{2} + 27 j^{2}}\notag \\
	& + 4  \sum_{i, j \geq 1 } (-1)^{i^2+j^2}q^{8i^{2} +  27 j^{2}} + 4  \sum_{i, j \geq 1}(-1)^{i^2}q^{ 8 i^{2} +  j^{2}}+ 
	4 \sum_{n \geq 1}q^{2n^2} + 4\sum_{i,j \geq 1}q^{2(i^2+j^2)} \pmod{8}.
	\end{align}
	To prove  \eqref{cong-24}, we need to look into the power series expansion from equation \eqref{seq4new}  Let's use the fact that the square of any integer can belong to either one of the following residue classes $\{0, 1, 4, 9,13,16,25,28\} \pmod{36}$. 	For every $i^2, j^2 \in \{0, 1, 4, 9,13,16,25,28\}$, it is very easy to check that $27i^2, i^2+27 j^2, 8i^2+27j^2, 8i^2 +j^2$ are not equivalent to $15$ in modulo $36$. This completes the proof.
\end{proof}
\begin{remark}
	We proved new congruence in \eqref{cong-24}, although the congruence occurs indirectly in this \cite{AMS} paper, our proof technique is with the mere usage of theta function identities.
\end{remark}
\section{Congruences for $\overline{R_{16,81}}(n)$ }\label{(16, 81)-result}
\begin{theorem}
	For all $n\geq 0$, we have
	\begin{align}
	\overline{R_{16,81}}(36n + 33)&\equiv 0 \pmod {8},\label{cong-25}
	\end{align}
\end{theorem}
\begin{proof}[Proof of Equation \eqref{cong-25}]
	Our strategy for proving the congruence in \eqref{cong-25} is same as the one we used for proving the congruences in \eqref{cong-10}, \eqref{cong-11}, \eqref{cong-12}. We begin by re-writing the generating function of $(16, 81)$- regular overpartitions 
	\begin{align}\label{pleqnew1}
	\sum_{n\geq 0}\overline{R_{16,81}}(n)q^n&= \frac{\varphi(-q^{16})\varphi(-q^{81})}{\varphi(-q)\varphi(-q^{1296})} 
	\end{align}
	Following the proof of \eqref{cong-10}, \eqref{cong-11}, \eqref{cong-12} we again utilize the identity $\frac{1}{\varphi(-q)} = \varphi(q)\varphi(q^2)^{2}\varphi(q^{4})^{4} \cdots$ and the fact that $\varphi(q^i)^j \equiv 1 \pmod{8} \text{ for all even } j \geq 4$ which helps us get to the following equation from equation \eqref{pleqnew1}
	\begin{align}\label{pleqnew3}
	&\frac{\varphi(-q^{16})\varphi(-q^{81})}{\varphi(-q)}\notag  \\
	&\equiv \varphi(-q^{16})\varphi(-q^{81})\varphi(q)\varphi(q^2)^2 \pmod{8}\notag\\
	&\equiv \biggl(1 + 2 \sum_{n \geq 1}(-q^{16})^{n^2}\biggl) \biggl(1 + 2 \sum_{n \geq 1}(-q^{81})^{n^2}\biggl)\biggl(1 + 2 \sum_{n \geq 1}q^{n^2}\biggl) \biggl(1 + 2 \sum_{n \geq 1}q^{2n^2}\biggl)^{2} \pmod{8}.
	\end{align}
	Expansion of the above equation gives us the following 
	\begin{align}\label{pleqnew2}
	&\equiv 1 + 2 \sum_{n \geq 1}q^{n^2}  + 2 \sum_{n \geq 1}(-q^{16})^{n^2} + 2 \sum_{n \geq 1}(-q^{81})^{n^2}+4 \sum_{n \geq 1}q^{2n^2} +4\sum_{i,j \geq 1}q^{2(i^2+j^2)}\notag\\
	& +4  \sum_{i, j \geq 1}q^{i^{2} + 16 j^{2}} + 4  \sum_{i, j \geq 1}q^{16i^{2} +  81 j^{2}} + 4  \sum_{i, j \geq 1}q^{i^{2} +  81 j^{2}}  \pmod{8}.
	\end{align}

	Here the argument that the exponents in the expansion can never be $\{33\} \pmod{12}$ follows in a similar way as was for proving the congruence in \eqref{cong-24}.
	Again by deeper analysis of the power series expansion from equation \eqref{pleqnew2} we can clearly see that the exponents of $q$ are either squares or $2$- times a square number or the sum of two squares. The square of any integer can belong to either one of the following residue classes $\{0, 1, 4, 9, 13, 16, 25, 28\} \pmod{36}$ and if we check the residues of sum of two squares modulo $36$, we will be able to see that the residue $33$ will never occur. This argument directly leads us to the fact that there will be no $q$ terms with the exponents of the form $36n + 33$.
\end{proof}
\begin{remark}
	It is important to note that this result will directly imply that $$ \overline{R}_{16, 81}(36n + 33) \equiv 0 \pmod{24},$$ given we have the proof of $\overline{R}_{16, 81}(36n + 33) \equiv 0 \pmod{3}$. 
\end{remark}
\section{Concluding Remarks}\label{Conclude}
We want to conclude by pointing out that some of our results immediately implies some of the congruences which were stated in the paper \cite{AMS}
\begin{align}
\overline{R}_{4, 3}(12n + 7) \equiv 0 \pmod{24}, \\
\overline{R}_{4, 3}(12n + 11) \equiv 0 \pmod{72}, \\
\overline{R}_{4, 9}(18n + 12) \equiv 0 \pmod{24}, \\
\overline{R}_{4, 9}(18n + 15) \equiv 0 \pmod{24}.
\end{align}
It will be interesting if one can come up with elementary proofs for the following congruences 
\begin{align}
\overline{R}_{8, 27}(36n + 15) \equiv 0 \pmod{3}, \\
\overline{R}_{16, 81}(36n + 33) \equiv 0 \pmod{6}.
\end{align}
If one can give elementary proofs of the above congruences then by combining our results in \eqref{cong-24} and \eqref{cong-25} we will have desirable elementary proofs for these congruences given in \cite{AMS}
\begin{align}
\overline{R}_{8, 27}(36n + 15) \equiv 0 \pmod{24},\\
\overline{R}_{16, 81}(36n + 33) \equiv 0 \pmod{48}.
\end{align}

It might be also possible to come up with proofs for congruences modulo higher powers of $2$ by using the method we used to prove the congruences modulo $8$, in case one knows the specific structure of $n$. For general $n$, without a specific structure one might need to check whether there are solutions to some of the quadratic Diophantine equations which occurs in the exponents. 

\section{Acknowledgement} 
Suparno Ghoshal was supported by the Deutsche Forschungsgemeinschaft (DFG, German Research Foundation) under Germany's Excellence Strategy -- EXC 2092 CASA -- 390781972.

\end{document}